\documentclass{article}

\usepackage{amssymb}
\usepackage{latexsym}
\usepackage{amsthm}
\usepackage{enumerate}
\usepackage{epsfig}
\usepackage{color}
\usepackage{float}
\usepackage{subfigure}
\usepackage[alphabetic,nobysame]{amsrefs}
\usepackage{amsmath}
\usepackage{ifthen}
\usepackage{makeidx}
\usepackage{lscape}
\usepackage{amscd}
\usepackage{tikz}
\usetikzlibrary{calc}

\def\e{\e}

\def\H{\mathbb{H}}

\def\Mt{\widetilde M}

\def\S{\Sigma}

\def\V{\mathcal{V}}

\def\W{\mathcal{W}}
\def\Xt{\widetilde{X}}
\def\e{\epsilon}

\def\half{\frac{1}{2}}

\def\ge{\geqslant}
\def\le{\leqslant}

\DeclareSymbolFont{AMSb}{U}{msb}{m}{n}
\DeclareMathSymbol{\Sph}{\mathbin}{AMSb}{"53}\DeclareMathSymbol{\Tor}{\mathbin}{AMSb}{"54}
\DeclareMathSymbol{\R}{\mathbin}{AMSb}{"52} \DeclareMathSymbol{\PP}{\mathbin}{AMSb}{"50}
\DeclareMathSymbol{\T}{\mathbin}{AMSb}{"54} \DeclareMathSymbol{\Z}{\mathbin}{AMSb}{"5A}
\DeclareMathSymbol{\C}{\mathbin}{AMSb}{"43} \DeclareMathSymbol{\E}{\mathbin}{AMSb}{"45}
\DeclareMathSymbol{\K}{\mathbin}{AMSb}{"4B}\DeclareMathSymbol{\N}{\mathbin}{AMSb}{"4E}
\newcommand{\twid}[1]{\widetilde{#1}}

\def\T{{\cal{T}}}

\def\V{{\cal{V}}}

\newcommand{\norm}[1]{\left|#1\right|}

\newtheorem{theorem}{Theorem}[section]

\newtheorem{proposition}[theorem]{Proposition}

\theoremstyle{definition}

\newtheorem{claim}[theorem]{Claim}

\newtheorem{definition}[theorem]{Definition}

\def\inj{\text{inj}}

\restylefloat{figure}

\begin{document}

\title{Morse functions to graphs and topological complexity for
  hyperbolic $3$-manifolds} 
\author{Diane Hoffoss and Joseph Maher}
\date{\today}

\maketitle

\tableofcontents

\begin{abstract}
Scharlemann and Thompson define the width of a 3-manifold $M$ as a
notion of complexity based on the topology of $M$.  Their original
definition had the property that the adjacency relation on handles
gave a linear order on handles, but here we consider a more general
definition due to Saito, Scharlemann and Schultens, in which the
adjacency relation on handles may give an arbitrary graph.  We show
that for compact hyperbolic $3$-manifolds, this is linearly related to
a notion of metric complexity, based on the areas of level sets of
Morse functions to graphs, which we call Gromov area.
\end{abstract}

\section{Introduction}

Mostow rigidity implies that for hyperbolic $3$-manifolds, the
hyperbolic metric is a topological invariant, so one might hope that
the topological and metric complexities are related. We shall show
that this is indeed the case for certain definitions of topological
and metric complexity. We first describe the notions of complexity we
shall use, and then give a brief outline of the arguments used to
relate topological and metric complexity in the subsequent sections.
In \cite{hoffoss-maher} we considered the linear version of these
invariants, while in this paper we consider the more general case of
invariants constructed from maps to graphs.  It will be convenient to
work with the collection of hyperbolic $3$-manifolds which are
complete, but do not contain cusps, and are not necessarily of finite
volume.  The reason for not considering manifolds with cusps is that
in the cusped case the surfaces separating the $3$-dimensional regions
in the topological decomposition we construct from the metric might
have essential intersection with the cusps.  In other words, if the
cusps are truncated to form a hyperbolic manifold with torus boundary
components, then the dividing surfaces may be surfaces with essential
boundary components on the boundary tori.  However, the currently
available versions of the topological decomposition results we use,
due to Scharlemann-Thompson and Saito-Scharlemann-Schultens, assume
that the dividing surfaces are closed.

This paper is not entirely self-contained, and relies on the results
of \cite{hoffoss-maher}, however we review the main definitions and
results from \cite{hoffoss-maher} for the convenience of the reader.

\subsection{Metric complexity}


In \cite{hoffoss-maher}, we considered the following definition of
metric complexity.  Let $M$ be a closed Riemannian $3$-manifold, and
let $f \colon M \to \R$ be a Morse function, i.e. $f$ is a smooth
function, all critical points are non-degenerate, and distinct
critical points have distinct images in $\R$. We define the
\emph{area} of $f$ to the maximum area of any level set
$F_t = f^{-1}(t)$ over all points $t \in \R$.  We define the
\emph{Morse area} of $M$ to be the infimum of the area of all Morse
functions $f \colon M \to \R$.

More generally, we may consider maps $f \colon M \to X$, where $X$ is
a trivalent graph.  Recall that for a Morse function $f \colon M \to
\R$ there are singularities of index $0, 1, 2$ and $3$. The
singularities of index $0$ and $3$ are known as birth or death
singularities respectively, and the level set foliation near the
singular point in $M$ is locally homeomorphic to the level sets of the
function $x^2 + y^2 + z^2$ close to the origin in $\R^3$. For
singularities of index $2$ and $3$, the level sets near the singular
point in $M$ are locally homeomorphic to the level sets of the
function $x^2 + y^2 - z^2$ close to the origin in $\R^3$.  

In the case of index $2$ or $3$, there is a map from a small open ball
containing the singular point to the leaf space of the level set
foliation. As the singular leaf divides a small ball about the
singular point into three connected components, the leaf space is a
trivalent graph with a single vertex and three edges, and we call such
a map a \emph{trivalent singularity}. If $X$ is a trivalent tree, we
say a map $f \colon M \to X$ is \emph{Morse} if it is a Morse function
on the interior of each edge of $X$, and at each trivalent vertex $v$
of $X$ the pre-image under $f$ is locally homeomorphic to a trivalent
singularity.  We say that the area of $f$ is the maximum area of
$F_t$, as $t$ runs over all points $t \in X$. The \emph{Gromov area}
of $M$ is the infimum of the area of $f \colon M \to X$ over all
trivalent graphs $X$, and all Morse functions $f \colon M \to X$.


This definition of metric complexity is a variant of Uryson width,
studied by Gromov in \cite{gromov}, though we consider the area of the
level sets instead of the diameter.  Alternatively, one may consider
it to be a variant of the definition of the waist of a manifold, but
we prefer to call it area, as the dimension of our spaces is fixed,
and the fibers have dimension two.

\subsection{Topological complexity}

We now describe the notions of topological complexity we shall
consider. A \emph{handlebody} is a compact $3$-manifold with boundary,
homeomorphic to a regular neighborhood of a graph in
$\mathbb{R}^3$. Up to homeomorphism, a handlebody is determined by the
genus $g$ of its boundary surface.  Every $3$-manifold $M$ has a
\emph{Heegaard splitting}, which is a decomposition of the manifold
into two handlebodies. This immediately gives a notion of complexity
for a $3$-manifold, called the \emph{Heegaard genus}, which is the
smallest genus of any Heegaard splitting of the $3$-manifold.

There is a refinement of this, due to Scharlemann and Thompson
\cite{st}, which we now describe. A \emph{compression body} $C$ is a
compact $3$-manifold with boundary, constructed by gluing some number
of $2$-handles to one side of a compact (but not necessarily
connected) surface cross interval and capping off any resulting
$2$-sphere components with $3$-balls.  The side of the surface cross
interval with no attached $2$-handles is called the \emph{top
  boundary} of the compression body and denoted by $\partial_+ C$, and
any other boundary components are called the \emph{lower boundary} of
the compression body, and denoted by $\partial_- C$.  A \emph{linear
  generalized Heegaard splitting},\footnote{We warn the reader that
  these are often referred to as generalized Heegaard splittings in
  the literature; however we wish to distinguish them from a more
  general notion described subsequently, which is also occasionally
  referred to in the literature as a generalized Heegaard splitting.}
which we shall abbreviate to \emph{linear splitting}, is a
decomposition of a $3$-manifold $M$ into a linearly ordered sequence
of (not necessarily connected) compression bodies
$C_1, \ldots C_{2n}$, such that the top boundary of an odd numbered
compression body $C_{2i+1}$ is equal to the top boundary of the
subsequent compression body $C_{2i+2}$, and the lower boundary of
$C_{2i+1}$ is equal to the lower boundary of the previous compression
body $C_{2i}$.  Let $H_i$ be the sequence of surfaces consisting of
the top boundaries of the compression bodies $C_{2i-1}$ and
$C_{2i}$. The complexity $c(H_i)$ of the surface $H_i$ is the sum of
the genera of each connected component, and the complexity of the
linear splitting is the collection of integers $\{c(H_i)\}$, arranged
in decreasing order. We order these complexities with the
lexicographic ordering. The \emph{width} of the linear splitting is
the maximum value (i.e. the first value) of $c(H_i)$ in the collection
$\{c(H_i)\}$. The \emph{linear width} of a $3$-manifold $M$ is the
minimum width over all possible linear generalized Heegaard
splittings. As a Heegaard splitting is a special case of a linear
splitting, the Heegaard genus of $M$ is an upper bound for the linear
width of $M$.  A linear splitting which gives the minimum complexity
of all possible linear splittings is called the \emph{thin position}
linear splitting.

There is a further refinement of this, described in Saito, Scharlemann
and Schultens \cite{sss}. A \emph{graph generalized Heegaard
  splitting}, which we shall abbreviate to \emph{graph splitting}, and
is called a \emph{fork complex} in \cite{sss}, is a decomposition of a
compact $3$-manifold $M$ into compression bodies $\{ C_i \}$, such
that for each compression body $C_i$, there is a compression body
$C_j$ such that the top boundary of $C_i$ is equal to the top boundary
of $C_j$. Furthermore, for each component of the lower boundary of
$C_i$, there is a compression body $C_k$, such that that component of
the lower boundary of $C_i$ is equal to a component of the lower
boundary of $C_k$. We emphasize that different components of the lower
boundary of $C_i$ may be attached to lower boundary components of
different compression bodies. Let $\{ H_i \}$ be the collection of top
boundary surfaces. The complexity of the graph splitting is the
collection of integers $\{c(H_i)\}$, arranged in decreasing
order. Again, we put the lexicographic ordering on these
complexities. A graph splitting which realizes the minimum complexity
is called a \emph{thin position graph splitting}. The \emph{width} of
the graph splitting is the maximum integer (i.e the first integer)
that appears in the complexity. The \emph{graph width} of a
$3$-manifold $M$ is the minimum width over all possible graph
splittings of $M$.  As a linear splitting is a special case of a graph
splitting, the linear width of $M$ is an upper bound for the graph
width of $M$.  The graph corresponding to the graph splitting is the
graph whose vertices are compression bodies, with edges connecting
pairs of compression bodies with common boundary components.

\subsection{Results}

In order to bound metric complexity in terms of topological complexity
we shall assume the following result announced by Pitts and Rubinstein
\cite{pr} (see also Rubinstein \cite{rubinstein}).

\begin{theorem} \cites{pr,rubinstein} \label{conjecture:pr} %
Let $M$ be a Riemannian $3$-manifold with a strongly irreducible
Heegaard splitting. Then the Heegaard surface is isotopic to a
minimal surface, or to the boundary of a regular neighborhood of a
non-orientable minimal surface with a small tube attached vertically
in the I-bundle structure.
\end{theorem}

A full proof of this result has not yet appeared in the literature,
though recent progress has been made by Colding and De Lellis
\cite{cdl}, De Lellis and Pellandrini \cite{dlp}, and Ketover
\cite{ketover}.

In \cite{hoffoss-maher} we showed:

\begin{theorem} \label{theorem:linear} %
There is a constant $K > 0$, such that for any closed hyperbolic
$3$-manifold,
\begin{equation} 
K ( \text{linear width}(M) ) \le \text{Morse area}(M) \le 4 \pi
(\text{linear width}(M)), 
\end{equation}
where the right hand bounds hold assuming Theorem
\ref{conjecture:pr}.
\end{theorem}

In this paper we show:

\begin{theorem} \label{theorem:graph} %
There is a constant $K > 0$, such that for any closed hyperbolic
$3$-manifold,
\begin{equation}
 K ( \text{graph width}(M) ) \le \text{Gromov area}(M) \le 4 \pi
(\text{graph width}(M)), 
\end{equation}
where the right hand bounds hold assuming Theorem
\ref{conjecture:pr}.
\end{theorem}

We also expect there to be upper and lower bounds on topological
complexity in terms of Uryson width, i.e. using diameter instead of
area, but we do not expect them to be linear.

\subsection{Related work in $3$-manifolds}

It may be of interest to compare our results with recent work of
Brock, Minsky, Namazi and Souto \cite{bmns} on manifolds with bounded
combinatorics. Let $C_1, \ldots C_n$ be a finite collection of
homeomorphism types of compact $3$-manifolds with marked boundary,
which we shall refer to as \emph{model pieces}, and fix a metric on
each one. A $3$-manifold $M$ is said to have \emph{bounded
  combinatorics} if it is a union of (possibly infinitely many) model
pieces glued together by homeomorphisms along their boundaries, with
certain restrictions on the gluing maps, which we do not describe in
detail here. In particular, a manifold with bounded combinatorics is a
manifold of bounded topological width. They show that such a manifold
$M$ is hyperbolic, with a lower bound on the injectivity radius, and
the hyperbolic metric is $K$-bilipshitz homeomorphic to the induced
metric on $M$ arising from the metrics on the model pieces. A choice
of foliation with compact leaves, containing the boundary leaves, on
each model piece then shows that the metric complexity is linearly
related to the topological complexity for this class of manifolds,
where the linear constants depend on the collection of model pieces.

Note that in our context, a bound on the topological width of the
manifold implies that the manifold is a union of compression bodies of
bounded genus, and there are finitely many of these up to
homeomorphism.  Their result assumes restrictions on the gluing maps,
but then shows the resulting manifold is hyperbolic, but the
bilipshitz constant $K$ depends on the width of $M$, i.e the genus of
the compression bodies. We assume that the manifold $M$ is compact and
hyperbolic, and make no restriction on the gluing maps between the
compression bodies, but we show that the linear constants relating
topological and metric complexities are independent of the genus of
the compression bodies.

\subsection{Outline}

In \cite{hoffoss-maher} we considered the linear case, in which the
range of the Morse function $f \colon M \to \R$ is $\R$.  Such a Morse
function has the property that for each $t \in \R$, the pre-image
$f^{-1}(t)$ is compact and separating.  For the case in which the
range of the Morse function $f \colon M \to X$ is a graph, one may
consider the lifted Morse function
$\widetilde f \colon \widetilde M \to \widetilde X$, where
$\widetilde X$ is the universal cover of $X$, and $\widetilde M$ is
the corresponding cover of $M$.  This lifted Morse function has the
property that for each $t \in \widetilde X$, each pre-image
$\widetilde f^{-1}(x)$ is compact and separating, and so many of the
arguments from \cite{hoffoss-maher} go through directly in this case.
In particular, we construct polyhedral approximations to the level
sets of $\widetilde f$, and show that they have bounded topological
complexity, as we now describe.

A choice of Margulis constant $\mu$ determines a thick-thin
decomposition for $M$, in which the thin part is a disjoint union of
Margulis tubes.  We also choose a Voronoi decomposition determined by
a maximal $\e$-separated collection of points in $M$.  This implies
that every Voronoi cell has diameter at most $\e$, and, given $\mu$,
we may choose $\e$ small enough such that every Voronoi cell that
intersects the thick part contains an embedded ball of radius $\e/2$.
The thick-thin decomposition of $M$, and the Voronoi decomposition of
$M$, lift to thick-thin decompositions and Voronoi decompositions of
the cover $\widetilde M$.  We give the details of this construction in
Sections \ref{section:tree}, \ref{section:voronoi} and
\ref{section:thickthin}.

A separating surface $F$ in $\widetilde M$ determines a partition of
the Voronoi cells, depending on which side of the surface the majority
of the volume of the (metric) ball of radius $\e/2$ inside the Voronoi
cell lies.  We will call the boundary between these two sets of
Voronoi cells a \emph{polyhedral surface} $S$, which is a union of
faces of Voronoi cells, and we can think of this as a combinatorial
approximation to the original surface $F$.

A key observation from \cite{hoffoss-maher} is that the number of
faces of the polyhedral surface in the thick part is bounded by the
area of $F$.  This is because in the thick part of $M$, the metric
ball of radius $\e/2$ in each Voronoi cell is embedded, so moving the
ball along a geodesic connecting the centers of the two Voronoi
produces at some point a metric ball whose volume is divided exactly
in two, giving a lower bound to the area of $F$ near that point.
There are bounds on the number of vertices and edges of any Voronoi
cell in terms of $\e$, so a bound on the number of faces of $S$ in the
thick part gives a bound on the Euler characteristic of $S$.  We are
unable to control the number of faces in the thin part, so we cap off
the part of $S$ in the thick part with surfaces of bounded Euler
characteristic contained in the thin part.  This produces surfaces of
bounded genus, which we call \emph{capped surfaces}.

In this way, the lift of a Morse function
$\widetilde f \colon \widetilde M \to \widetilde X$ gives rise to a
collection of polyhedral surfaces in $\widetilde M$ of bounded genus.
These surfaces are constant except at finitely many points of
$\widetilde X$, which we call \emph{cell splitters}, where a level set
divides the ball contained in a Voronoi cell exactly in half.  We give
the details of the construction of the capped surfaces and the
properties of the cell splitters in Sections \ref{section:splitter}
and \ref{section:capped}.

The key step, in Section \ref{section:equivariant}, is to show that we
may construct these surfaces equivariantly in $\widetilde M$, so they
project down to embedded surfaces in $M$, with the same bounds on
their topological complexity.

Finally, in Section \ref{section:bounded}, by considering the local
configuration near a cell splitter, we show that the regions between
the capped surfaces may be constructed using a number of handles
bounded in terms of the area of the level sets $\widetilde f^{-1}(t)$,
and so this the bounds topological complexity of the decomposition of
$M$ given by the capped surfaces in terms of metric complexity of $M$.

The bound in the other direction is a direct consequence of the bound
from \cite{hoffoss-maher}, though we review the argument in the
Section \ref{section:sss} for the convenience of the reader.

\subsection{Acknowledgements}

The authors would like to thank Dick Canary, David Futer, David Gabai,
Joel Hass, Daniel Ketover, Sadayoshi Kojima, Yair Minksy, Yo'av Rieck
and Dylan Thurston for helpful conversations, and the Tokyo Institute
of Technology for its hospitality and support. The second author was
supported by the Simons Foundation and PSC-CUNY.  This material is
based upon work supported by the National Science Foundation under
Grant No.~DMS-1440140 while the second author was in residence at the
Mathematical Sciences Research Institute in Berkeley, California,
during the Fall 2016 semester.

\section{Gromov area bounds graph width} 
\label{section:metric bounds topology}

In this section we show that we can bound the topological complexity of
the manifold in terms of its metric complexity, i.e.  we show that
graph width is bounded in terms of Gromov area.

\begin{theorem}
There is a constant $K$, such that for any closed hyperbolic
$3$-manifold $M$,
\[ \text{graph width}(M) \le K ( \text{Gromov area}(M) ). \]
\end{theorem}

Let $f \colon M \to X$ be a Morse function onto a graph $X$, such that
the Gromov area of $f$ is arbitrarily close to the Gromov area of $M$.
Any metric graph is arbitrarily close to a trivalent metric graph, so
we may assume the graph is trivalent.  We now show that we may assume
the level sets of $f$ are connected.

\begin{proposition}
Let $M$ be a Riemannian manifold, and let $f \colon M \to X$ be a
Morse function onto a trivalent graph $X$. Then there is a trivalent
graph $X'$, and a Morse function $f' \colon M \to X'$ with connected
level sets, with $\text{Gromov area}(f') \le \text{Gromov area}(f)$.
\end{proposition}

\begin{proof}
The level sets of the function $f$ give a singular foliation of $M$
with compact leaves, which we shall call the \emph{level set
  foliation}, and the leaves of this foliation are precisely the
connected components of the pre-images of points in $M$.  Consider the
leaf space $L$ of the level set foliation, i.e. the space obtained
from $M$ by identifying points in the same leaf. As all leaves are
compact, the leaf space is Hausdorff. The leaf space is a trivalent
graph, with vertices corresponding to vertex singularities, and the
maximum area of the pre-images of the quotient map is less than or
equal to the maximum area of the pre-images of $f$.  Therefore, we may
choose $f'$ to be the leaf space quotient map $f' \colon M \to L$,
which is a Morse function onto a trivalent graph, and has connected
level sets, with the property that the area of the level sets of $f'$
is bounded by the area of the level sets of $f$.
\end{proof}

In particular, this means that the vertices of $X$ are precisely the
critical points of the Morse function $f$ in which a connected level
set splits into two connected components.

\subsection{Morse functions to trees} 
\label{section:tree}

We would like to work in the cover $\widetilde M$ of $M$ corresponding
to the universal cover $\widetilde X$ of the graph $X$, which will
have the key advantage that all pre-image surfaces are separating in
$\widetilde M$.  In fact, the induced map on fundamental groups $f_*
\colon \pi_1 M \to \pi_1 X$ is surjective, but as we do not use this
property, we omit the proof.

Let $p \colon \twid{M} \rightarrow M$ be the cover of $M$
corresponding to the kernel of the induced map $f_* \colon \pi_1 M
\rightarrow \pi_1 X $, and let $c:\twid{X}\rightarrow X$ be the
universal cover of $X$, so $\twid{X}$ is a tree. Then the map $f \circ
p \colon \twid{M} \rightarrow X$ lifts to a map $h = \twid{f \circ p}
: \twid{M} \rightarrow \twid{X}$.  Since each leaf $F_t$ in $M$ maps
to a single point in $X$, the fundamental group of each leaf is
contained in $\ker(f)$.  Therefore, each leaf in $M$ lifts to a leaf
in $\twid{M}$, and as the cover is regular, the pre-image of a point
$t \in \Xt$ is a disjoint union of homeomorphic copies of
$F_{c(t)}$. In particular, the area bound for the leaves $F_t$ in
$M$ is also an area bound for the leaves $H_t = h^{-1}(t)$ in
$\twid{M}$.

\begin{center}
\begin{tikzpicture}[node distance=2cm, auto]
  \node (A) {$\twid{M}$};
  \node (B) [right of=A] {$\twid{X}$};
  \node (C) [below of=A] {$M$};
  \node (D) [below of=B] {$X$};
  \draw[->] (A) to node {$h = \twid{f \circ p}$} (B);
  \draw[->] (A) to node {$p$} (C);
  \draw[->] (B) to node {$c$} (D);
  \draw[->] (C) to node {$f$} (D);
\end{tikzpicture}
\end{center}

As $\Xt$ is a tree, every point is separating, and so every pre-image
surface $H_t = h^{-1}(t)$ is also separating.

\subsection{Voronoi cells} \label{section:voronoi}

We will approximate the level sets of $f$ by surfaces consisting of
faces of Voronoi cells.  We now describe in detail the Voronoi cell
decompositions we shall use, and their properties. The definitions in
this section are taken verbatim from \cite{hoffoss-maher}, but we
include them in this section for the convenience of the reader.

A \emph{polygon} in $\H^3$ is a bounded subset of a hyperbolic plane
whose boundary consists of a finite number of geodesic segments. A
\emph{polyhedron} in $\H^3$ is a convex topological $3$-ball in $\H^3$
whose boundary consists of a finite collection of polygons. A
\emph{polyhedral cell decomposition} of $\H^3$ is a cell decomposition
in which which every $3$-cell is a polyhedron, each $2$-cell is a
polygon, and the edges are all geodesic segments. We say a cell
decomposition of a complete hyperbolic manifold $M$ is
\emph{polyhedral} if its pre-image in the universal cover $\H^3$ is
polyhedral.

Let $X = \{ x_i \}$ be a discrete collection of points in
$3$-dimensional hyperbolic space $\H^3$. The Voronoi cell $V_i$
determined by $x_i \in X$ consists of all points of $M$ which are
closer to $x_i$ than any other $x_j \in X$, i.e.
\[ V_i = \{ x \in \H^3 \mid d(x, x_i) \le d(x, x_j) \text{ for all }
x_j \in \Xt \}. \]
We shall call $x_i$ the \emph{center} of the Voronoi cell $V_i$, and
we shall write $\V = \{ V_i \}$ for the collection of Voronoi cells
determined by $X$. Voronoi cells are convex sets in $\H^3$, and hence
topological balls.  The set of points equidistant from both $x_i$ and
$x_j$ is a totally geodesic hyperbolic plane in $\H^3$.  A \emph{face}
$\Phi$ of the Voronoi decomposition consists of all points which lie
in two distinct Voronoi cells $V_i$ and $V_j$, so $\Phi$ is contained
in a geodesic plane. An \emph{edge} $e$ of the Voronoi decomposition
consists of all points which lie in three distinct Voronoi cells
$V_i, V_j$ and $V_k$, which is a geodesic segment, and a \emph{vertex}
$v$ is a point lying in four distinct Voronoi cells $V_i, V_j, V_k$
and $V_l$.  By general position, we may assume that all edges of the
Voronoi decomposition are contained in exactly three distinct faces,
the collection of vertices is a discrete set, and there are no points
which lie in more than four distinct Voronoi cells. We shall call such
a Voronoi decomposition a \emph{regular} Voronoi decomposition, and it
is a polyhedral decomposition of $\H^3$.  As each edge is $3$-valent,
and each vertex is $4$-valent, this implies that the dual cell
structure is a simplicial triangulation of $\H^3$, which we shall
refer to as the \emph{dual triangulation}. The dual triangulation may
be realised in $\H^3$ by choosing the vertices to be the centers $x_i$
of the Voronoi cells and the edges to be geodesic segments connecting
the vertices, and we shall always assume that we have done this. In
this case the triangles and tetrahedra are geodesic triangles and
tetrahedra in $\H^3$.

Given a collection of points $X = \{ x_i \}$ in a hyperbolic
$3$-manifold $M$, let $\Xt$ be the pre-image of $X$ in the universal
cover of $M$, which is isometric to $\H^3$. As $\Xt$ is equivariant,
the corresponding Voronoi cell decomposition $\V$ of $\H^3$ is also
equivariant. The distance condition implies that the interior of each
Voronoi cell $V$ is mapped down homeomorphically by the covering
projection, though the covering projection may identify faces, edges
or vertices of $V_i$ under projection into $M$.  By abuse of notation,
we shall refer to the resulting polyhedral decomposition of $M$ as the
Voronoi decomposition $\V$ of $M$.  By general position, we may assume
that $\V$ is regular.  The dual triangulation is also equivariant, and
projects down to a triangulation of $M$, which we will also refer to
as the dual triangulation, though this triangulation may no longer be
simplicial.

We shall write $B(x, r)$ for the closed metric ball of radius $r$
about $x$ in $M$, i.e.
\[  B(x, r) = \{ y \in M \mid d(x, y) \le r \}.   \]
A metric ball in $M$ need not be a topological ball in general.  We
shall write $\inj_M(x)$ for the injectivity radius of $M$ at $x$,
i.e. the radius of the largest embedded ball in $M$ centered at $x$.
Then the injectivity radius of $M$, denoted $\inj(M)$, is defined to
be
\[ \inj(M) = \inf_{x \in M} \inj_M(x). \]

We say a collection $\{ x_i \}$ of points in $M$ is
\emph{$\epsilon$-separated} if the distance between any pair of points
is at least $\epsilon$, i.e. $d(x_i, x_j) \ge \epsilon$, for all $i
\not = j$. Let $\{ x_i \}$ be a maximal collection of
$\epsilon$-separated points in $M$, and let $\V$ be the corresponding
Voronoi cell division of $M$.  Since the collection $\{ x_i \}$ is
maximal, each Voronoi cell is contained in a metric ball of radius
$\e$ about its center. Furthermore, if the injectivity radius at the
center $x_i$ is at least $2\e$, then as the points $x_i$ are distance
at least $\e$ apart, each Voronoi cell contains a topological ball of
radius $\e/2$ about its center, i.e.
\[ B(x_i, \epsilon/2 )  \subset V_i \subset B(x_i, \epsilon).     \]

\begin{definition}
Let $M$ be a complete hyperbolic $3$-manifold.  We say a Voronoi
decomposition $\V$ is $\e$-regular, if it is regular, and it arises
from a maximal collection of $\e$-separated points.
\end{definition}

A \emph{simple arc} in the boundary of a tetrahedron is a properly
embedded arc in a face of the tetrahedron with endpoints in distinct
edges. A \emph{triangle} in a tetrahedron is a properly embedded disc
whose boundary is a union of three simple arcs, and a
\emph{quadrilateral} is a properly embedded disc whose boundary is the
union of four simple arcs. A \emph{normal surface} in a triangulated
$3$-manifold is a surface that intersects each tetrahedron in a union
of normal triangles and quadrilaterals.

One useful property of $\e$-regular Voronoi decompositions is that the
boundary of any union of Voronoi cells is an embedded surface, in fact
an embedded normal surface in the dual triangulation.

\begin{proposition} \cite{hoffoss-maher}*{Proposition 2.2} %
Let $M$ be a complete hyperbolic manifold without cusps, and let $\V$
be an $\e$-regular Voronoi decomposition. Let $P$ be a union of
Voronoi cells in $\V$, and let $S$ be the boundary of $P$. Then $S$ is
an embedded surface in $M$.
\end{proposition}

In \cite{hoffoss-maher} this result is stated for compact hyperbolic
$3$-manifolds, but the proof works for complete hyperbolic
$3$-manifolds without cusps.

We shall say a Voronoi cell $V_i$ with center $x_i$ is an
\emph{$\e$-deep} Voronoi cell if the injectivity radius at $x_i$ is at
least $4\e$, i.e. $\inj_M(x_i) \ge 4\e$, and in particular this
implies that the metric ball $B(x_i, 3\e)$ is a topological ball. We
shall also call centers, faces, edges and vertices of $\e$-deep
Voronoi cells $\e$-deep.
In the next section we will choose a fixed $\e$ independent of the
manifold $M$, and we will just say \emph{deep} instead of
$\e$-deep. We shall write $\W$ for the subset of $\V$ consisting of
deep Voronoi cells. If $\e < \tfrac{1}{4}\inj(M)$, then $\V = \W$ and
all Voronoi cells are deep.

Finally, we recall that there are bounds, which only depend on
$\epsilon$, on the number of vertices, edges and faces of a deep
Voronoi cell.

\begin{proposition} \cite{hoffoss-maher}*{Proposition 2.3}\label{prop:bound} %
Let $M$ be a complete hyperbolic $3$-manifold with an $\e$-regular
Voronoi decomposition $\V$, and let $\W$ be the collection of deep
Voronoi cells.  Then there is a number $J$, which only depends on
$\e$, such that each deep Voronoi cell $W_i \in \W$ has at most $J$
faces, edges and vertices.
\end{proposition}

Again, in \cite{hoffoss-maher}, these results are stated for compact
hyperbolic $3$-manifolds, but the proofs work for complete hyperbolic
$3$-manifolds without cusps.

\subsection{Margulis tubes} \label{section:thickthin}

We will use the Margulis Lemma and the \emph{thick-thin} decomposition
for finite volume hyperbolic $3$-manifolds, and we now review these
results.

Given a number $\mu > 0$, let $X_\mu = M_{[\mu, \infty)}$ be the
\emph{thick part} of $M$, i.e. the union of all points $x$ of $M$ with
$\inj_M(x) \ge \mu$.  We shall refer to the closure of the complement
of the thick part as the \emph{thin part} and denote it by $T_\mu =
\overline{M \setminus X}$.

The Margulis Lemma states that there is a constant $\mu_0 > 0$, such
that for any compact hyperbolic $3$-manifold, the thin part is a
disjoint union of solid tori, called \emph{Margulis tubes}, and each
of these solid tori is a regular metric neighborhood of an embedded
closed geodesic of length less than $\mu_0$.  In the case in which $M$
is complete without cusps, there is an extra possibility, as a
component of the thin part may also be the universal cover of such a
solid torus, and we shall refer to such a component as an
\emph{infinite Margulis tube}.  We shall call a number $\mu_0$ for
which this result holds a \emph{Margulis constant} for
$\mathbb{H}^3$. If $\mu_0$ is a Margulis constant for $\mathbb{H}^3$,
then so is $\mu$ for any $0 < \mu < \mu_0$, and furthermore, given
$\mu$ and $\mu_0$ there is a number $\delta > 0$ such that
$N_{\delta}(T_{\mu}) \subseteq T_{\mu_0}$.  For the remainder for this
section we shall fix a pair of numbers $(\mu, \e)$ such that there are
Margulis constants $0 < \mu_1 < \mu < \mu_2$, a number $\delta$ such
that $N_{\delta}(T_{\mu}) \subseteq T_{\mu_2} \setminus T_{\mu_1}$,
and $\e = \tfrac{1}{4} \min \{ \mu_1, \delta \}$. We shall call $(\mu,
\e)$ a choice of \emph{MV}-constants for $\mathbb{H}^3$.

Let $(\mu, \e)$ be a choice of $MV$-constants, and consider an
$\e$-regular Voronoi decomposition of $M$. The fact that
$N_{\delta}(T_{\mu}) \subseteq T_{\mu_2} \setminus T_{\mu_1}$ means
that we may adjust the boundary of $T_{\mu}$ by an arbitrarily small
isotopy so that it is transverse to the Voronoi cells, and we will
assume that we have done this for the remainder of this section.  Our
choice of $\e$ implies that the thick part $X_\mu$ is contained in the
Voronoi cells in the deep part, i.e. $X_\mu \subset \bigcup_{W_i \in
  \W} W_i$, so in particular $\partial X_\mu = \partial T_\mu$ is
contained in the deep part. Furthermore, each deep Voronoi cell hits
at most one component of $T_\mu$.

\subsection{Cell splitters}\label{section:splitter}

The polyhedral surfaces we construct will be constant, except for a
discrete collection of points in $Y$, which roughly speaking
correspond to points $t \in Y$ for which the level set $f^{-1}(t)$
divide a Voronoi cell in half.  For technical reasons, we use points
which divide a ball of fixed size in the Voronoi cell in half, as we
now describe.

Let $t$ be a point in a trivalent tree $Y$. We shall write $Y_t^{c_i}$
for the closures of the connected components of $Y \setminus t$, and
we shall call these the \emph{complements} of $t$. If $t$ lies in the
interior of an edge, then there are precisely two complements, while
if $t$ is a vertex, there are precisely three complements.

Let $M$ be a complete hyperbolic $3$-manifold without cusps, and let
$h \colon M \rightarrow Y$ be a Morse function onto a trivalent tree
$Y$.  Given $t \in Y$, let $H_t^{c_i} = h^{-1}(Y_t^{c_i})$, and we
shall refer to these as the \emph{complements} of $H_t$ in $M$. As
before, there are either two or three complementary regions depending
on whether $t$ lies in the interior of an edge, or is a vertex in $Y$.

\begin{definition} \label{definition:splitter} %
Let $M$ be a complete hyperbolic $3$-manifold without cusps, with an
$\e$-regular Voronoi decomposition $\V$.  Let $h \colon M \to Y$ be a
Morse function to a tree $Y$, and let $V$ be a Voronoi cell with
center $x$. Suppose that a point $t \in Y$ has the property that for
each complementary region $H_t^{c_i}$, the volume of
$H_t^{c_i} \cap B(x, \epsilon/2) \cap V$ is at most half the volume of
the topological ball $B(x, \epsilon/2) \cap V$.  Then we say that $t$
is a \emph{cell splitter} for the Voronoi cell $V$.
\end{definition}

\begin{proposition}
Let $M$ be a complete hyperbolic $3$-manifold without cusps, with an
$\e$-regular Voronoi decomposition $\V$.  Let $h \colon M \to Y$ be a
Morse function to a tree, and let $V$ be a Voronoi cell with center
$x$.  Then there is a unique cell splitter $t \in Y$ for $V$.
\end{proposition}

\begin{proof}
We first show existence.  Let $B$ be the topological ball $B(x, \e/2)
\cap V$, and let $v$ be the volume of this ball. Consider $h(B)
\subset Y$.  If there is a vertex of $Y$ which is a cell splitter,
then we are done.  Otherwise, suppose no vertex of $h(B)$ is a cell
splitter. If $t$ is a vertex in $h(B)$ which is not a cell splitter,
then there is at least one complementary region $Y_t^{c_i}$ such that
$H_t^{c_i} \cap B(x, \e/2) \cap V$ has volume more than
$\tfrac{1}{2}v$, and $Y_t^{c_i} \cap h(B)$ has at least one fewer
vertex. So proceeding by induction, we may reduce to the case in which
$h(B)$ contains an interval $I$ with no vertices such that $h^{-1}(I)
\cap B(x, \e/2) \cap V$ has volume at least $\tfrac{1}{2} v$.  In this
case, let $t_0$ and $t_1$ be the endpoints of $I$, and consider
$h^{-1}([t_0, s])$, for $s \in I$. When $s = t_0$, this has volume
less than $\tfrac{1}{2} v$, and has volume greater than
$\tfrac{1}{2}v$ when $s = t_1$. As the volume changes continuously
with $s$, there is a point $t'$ such that $H_{t'}$ divides $B$ into
two regions, each of which has volume exactly $\tfrac{1}{2}v$, so $t'$
is a cell splitter for $V$.

We now show uniqueness.  First suppose $t$ is a cell splitter which is
not a vertex. Then there are precisely two complementary regions
$H_t^{c_1}$ and $H_t^{c_2}$, each of which must have exactly half the
volume of $B(x, \e/2) \cap V$, and we shall denote this volume by
$v$. Any other point $t'$ has a complementary region which contains at
least one of these complements, and so has volume greater than
$\tfrac{1}{2} v$, and so can not be a cell splitter.

Finally suppose $t$ is a cell splitter which is a vertex. Then there
are three complements $H_t^{c_1}, H_t^{c_2}$ and $H_t^{c_3}$, each of
which has volume at most $\tfrac{1}{2}v$. As each region has volume at
most $\tfrac{1}{2}v$, any two regions must have total volume at least
$\tfrac{1}{2}v$. Any other point $t' \in Y$ must have a complementary
region which contains at least two of the complements of $H_t$, and so
has a complement with volume strictly greater than $\tfrac{1}{2}v$,
and so can not be a cell splitter.
\end{proof}

\begin{definition} \label{definition:generic} %
We say that a Morse function $f \colon M \to Y$ to a tree $Y$ is
\emph{generic} with respect to a Voronoi decomposition $\mathcal{V}$
if the cell splitters for distinct Voronoi cells $V_i$ correspond to
distinct points $t_i \in Y$. We say a point $t \in Y$ is
\emph{generic} if it is not a critical point for the Morse function,
and is not a cell splitter.
\end{definition}

We may assume that $f$ is generic by an arbitrarily small perturbation
of $f$, and we shall always assume that $f$ is generic from now on.
Finally, we remark that a trivalent vertex in $Y$ is not necessarily a
cell splitter.

\subsection{Polyhedral and capped surfaces} \label{section:capped}

Let $Q$ be a $3$-dimensional submanifold of a complete hyperbolic
$3$-manifold $M$ without cusps, with boundary an embedded separating
surface $F$. In this section we show how to approximate $Q$ by a union
of Voronoi cells, which in turn gives an approximation to $F$ by an
embedded surface $S$ which is a union of faces of Voronoi cells.

We say a region $R$ is \emph{generic} if for every Voronoi cell $V_i$
with center $x_i$, the region consisting of the intersection of
$B(x_i, \e/2)$ with the interior of $V_i$ does not have exactly half
its volume lying in $R$. We say a separating surface $F$ in $M$ is
\emph{generic} if it bounds a generic region.

Let $P$ be the collection of Voronoi cells for which at least half of
the volume of $B(x_i, \e/2) \cap \text{interior}(V_i)$ lies in $Q$. We
say the $P$ is the \emph{polyhedral region} determined by $Q$. The
polyhedral region $P$ may be empty, even if $Q$ is non-empty. The
boundary of $Q$ is a polyhedral surface $S$, which we shall call the
\emph{polyhedral surface} associated to $F = \partial Q$, and is a
normal surface in the dual triangulation.  We will use the following
bound on the number of faces and boundary components of the
intersection of the polyhedral surface $S$ with the thick part of the
manifold, in terms of the area of the corresponding surface $F$.  If
$S$ is a surface, we will write $\norm{\partial S}$ for the number of
boundary components of $S$, and if $S'$ is a subset of a polyhedral
surface $S$, we will write $\| S' \|$ for the number of faces of $S$
which intersect $S'$.

\begin{proposition} 
\cite{hoffoss-maher}*{Proposition 2.10, 2.13} \label{cor:bound} %
Let $(\mu, \e)$ be $MV$-constants, and let $M$ be a complete
hyperbolic $3$-manifold without cusps, with an $\e$-regular Voronoi
decomposition $\V$ with deep part $\W$ and thick part $X_\mu$. Then
there is a constant $K$, which only depends on the $MV$-constants,
such that for any generic embedded separating surface $F$ in $M$, the
corresponding polyhedral surface $S$ satisfies:
\[ \| S \cap X_\mu \| \le K \text{area}(F), \]
and
\[ \norm{\partial ( S \cap X_\mu ) } \le K \text{area}(F). \]
\end{proposition}

In \cite{hoffoss-maher}, this result is stated for the level set of a
Morse function $F$ on a compact hyperbolic manifold, and one may then
observe that every separating surface is the level set of some Morse
function, though in fact, the proof only uses the fact that $F$ is
separating.  In \cite{hoffoss-maher}*{Proposition 2.10} the bound is
stated in terms of $S \cap \W$.  However, as $S \cap X_\mu \subset S
\cap \W$, the stated bound follows immediately.

For a polyhedral surface $S$, each boundary component of the surface
$S \cap X_\mu$ is contained in $T_\mu$, so $S \cap X_\mu$ is a
properly embedded surface in $X_\mu$.  We now wish to cap off the
properly embedded surfaces $S \cap X_\mu$ with properly embedded
surfaces in $T_\mu$ to form closed surfaces.  We warn the reader that
the following definition differs slightly from the definition in
\cite{hoffoss-maher}, as we extend the definition to include the case
in which $T_\mu$ has infinite components.

\begin{definition}
A separating surface $F$ in $M$ gives rise to a polyhedral surface
$S$, which meets $\partial T_\mu$ transversely, and intersects
$\partial T_\mu$ in a collection of simple closed curves which is
separating in $\partial T_\mu$.  We replace $S$ inside the thin part
by surfaces $\{ U_i \}$ which we now describe.  For each torus
component $T_i$ in $\partial T_\mu$ choose a subsurface $U_i$ bounded
by $S \cap \partial T_i$.  For each infinite component $T_i$, choose a
not necessarily connected surface $U_i$ as follows: for each essential
curve in the annulus $\partial T_i$ choose a disc it bounds in $T_i$,
and then let $U_i$ be the union of these discs with the planar surface
bounded by the remaining inessential curves.  We call the resulting
surface a \emph{capped surface} $S^+ = (S \cap X_\mu) \cup \bigcup_i
U_i$.
\end{definition}

We will use the following property of the capped surfaces.

\begin{proposition} \label{prop:capped}
Let $(\mu, \e)$ be $MV$-constants, and let $M$ be a complete
hyperbolic $3$-manifold without cusps, with thin part $T_\mu$, and
with with an $\e$-regular Voronoi decomposition $\V$. Then there is a
constant $K$, which only depends on $\e$, such that for any generic
embedded separating surface $F$ in $M$, the corresponding capped
surface $S^+$ satisfies:
\[ \text{genus}(S^+) \le K \text{area}(F). \]
\end{proposition}

The proof of this result is essentially the same as the proof of
\cite{hoffoss-maher}*{Proposition 2.14}, and instead of repeating the
entire argument, we explain the minor extension needed. The only
difference is that \cite{hoffoss-maher}*{Proposition 2.14} is stated
for closed hyperbolic manifolds, whereas Proposition \ref{prop:capped}
is stated for complete hyperbolic manifolds without cusps, so the thin
part of $M$ may have infinite Margulis tubes.  This makes no
difference to the estimates of the number of faces and boundary
components of the resulting polyhedral surface in terms of the area of
the original surface.  The extension of the definition of capped
surface to the infinite case only involves capping off with planar
surfaces, so the same genus bounds hold.

\subsection{Disjoint equivariant surfaces}\label{section:equivariant}

Each collection of points $t_i$ in $Y$ corresponds to a collection
$S^+_i$ of capped surfaces.  In this section we show that if the
collection of points is equivariant, then we may arrange for the
capped surfaces to be disjoint and equivariant.

Let $M$ be a $3$-manifold which admits a group of covering
translations $G$. We say a subset $U \subset M$ is \emph{equivariant}
if it is preserved by $G$. We say a Voronoi decomposition $\V$ of $M$
is \emph{equivariant} if the centers of the Voronoi cells form an
equivariant set in $M$.

Let $W$ be an equivariant collection of points in $\widetilde X$, none
of which are either cell splitters or critical points of the Morse
function $h$. We say two points $t_i$, $t_j$ in $W$ are
\emph{adjacent} if the geodesic connecting them in the tree $\Xt$ does
not contain any other point of $W$. We may choose $W$ such that the
geodesic in $\Xt$ connecting any pair of adjacent points in $W$
contains either a single cell splitter, a single trivalent trivalent
vertex of $\Xt$, or neither of these two types of points.

Consider the collection $S$ of polyhedral surfaces $S_t$, as $t$ runs
over $W$.  As the collection $W$ is equivariant, $S$ is also
equivariant.  Note that although each surface in $S$ is individually
embedded, each surface in $S$ will share many common faces with other
surfaces in $S$.  We will now make this collection simultaneously
equivariantly disjoint, so that we may push them down to $M$ to obtain
a collection of disjoint surfaces which will act as our splitting
surfaces in a graph splitting of $M$.

\begin{proposition}
Let $M$ be a closed hyperbolic 3-manifold of injectivity radius at
least $2\e$, with an $\e$-regular Voronoi decomposition $\V$, and a
generic Morse function $f : M \rightarrow X$ onto a trivalent graph
$X$ with connected level sets. Let $p \colon \twid{M} \rightarrow M$
be the cover of $M$ corresponding to the kernel of the induced map
$f_* \colon \pi_1 M \rightarrow \pi_1 X $, and let
$c:\twid{X}\rightarrow X$ be the universal cover of $X$.  Let $W$ be a
discrete equivariant collection of points in $\widetilde X$.  Then the
collection of polyhedral surfaces $\{ S_w \mid w \in W \}$ in
$\twid{M}$ is equivariantly isotopic to a disjoint collection of
surfaces $\{ \S_w \mid w \in W \}$, and furthermore this equivariant
isotopy may be chosen to be supported in a neighborhood of the
2-skeleton of the induced Voronoi decomposition of $\widetilde M$.
\end{proposition}

\begin{proof}
We now give a recipe for constructing surfaces $\S_t$, for $t \in
W$. Each individual surface $\S_t$ will be isotopic to the original
$S_t$, but the union of the surfaces $\S_t$ will be equivariantly
disjointly embedded in $\twid{M}$.

We first show that there is a canonical ordering of the polyhedral
surfaces $\S_t$ which share a common face.  Let $\Phi$ be a face of a
Voronoi cell in $\widetilde M$, and let $V(x_1)$ and $V(x_2)$ be the
adjacent Voronoi cells. Let $t_1$ and $t_2$ be cell splitters for
$V(x_1)$ and $V(x_2)$, so that $H_{t_i} = h^{-1}(t_i)$ is the surface
which divides $B_{\e/2}(x_i)$ precisely in half, for $i = 1,2$.

We say a point in $\Xt$ is \emph{regular} if it is not a
cell splitter, and not a critical point for the Morse function $h$.

\begin{claim}
The collection of regular points in $\widetilde X$ corresponding to
polyhedral surfaces $\S_t$ which contain the face $\Phi$ is precisely
the regular points contained in the geodesic in $\Xt$ from $t_1$ to
$t_2$.
\end{claim}

\begin{proof}
The two embedded surfaces $H_{t_1}$ and $H_{t_2}$ divide $\Mt$ into
three parts; call them $A, B$ and $C$, with $A$ the part only hitting
$H_{t_1}$, and $B$ the part hitting both $H_{t_1}$ and $H_{t_2}$.

Let $\gamma$ be the geodesic in $\Xt$ from $t_1$ to $t_2$. Each point
$t$ in $\gamma$ corresponds to a surface $H_t$ dividing $\Mt$ at most
3 parts, one of which contains $A$, and another containing $C$. Let
$P_t$ be the part containing $A$.  Then, writing $\norm{A}$ for the
volume of a region $A$, 
\[ \norm{ B_{\e/2}(x_1) \cap P_t} \ge \norm{B_{\e/2}(x_1)
  \cap A} \ge \half \norm{ B_{\e/2}(x_1) }
\]
and 
\[ \norm{ B_{\e/2}(x_2)
  \setminus P_t} \ge \norm{B_{\e/2}(x_2) \setminus C} \ge \half \norm{
  B_{\e/2}(x_2) }.
\]
Therefore the two Voronoi cells $V(x_1)$ and $V(x_2)$ lie in different
partitions of the Voronoi cells determined by $t$, and so $\Phi$ lies
in the polyhedral surface $\S_t$.

Conversely, suppose $t$ does not lie on the path $\gamma$, then $t$
divides $\Xt$ into at most three parts, and $\gamma$ is contained in
exactly one of these parts. This means that $H_{t_1}$ and $H_{t_2}$
are contained in the same complementary component of $H_t$, and so $\Phi$
cannot be a face of $\S_t$.
\end{proof}

It suffices to show that we can isotope the normal surfaces,
preserving the fact that they are normal, so that they have disjoint
intersection in the $2$-skeleton of the dual triangulation.

Let $e$ be an edge of the dual triangulation, with vertices $x_1$ and
$x_2$, with corresponding cell splitters $t_1$ and $t_2$. A normal
surface $S_i$ intersects $e$ if and only if the corresponding point
$w_i$ lies in the geodesic $[t_1, t_2]$ in $\Xt$ connecting $t_1$ and
$t_2$. The points $w_i$ in $e$ therefore inherit an order from $[t_1,
t_2]$, and we may isotope the normal surfaces by a normal isotopy so
that they intersect the edge $e$ in the same order. As the interiors
of each edge have disjoint images under the covering translations, and
the collection of edges is equivariant, we may do this normal isotopy
equivariantly.

Let $\Phi$ be a triangle in the dual triangulation, with vertices
$x_1, x_2$ and $x_3$, and corresponding cell splitters $t_1, t_2$ and
$t_3$. As above, the collection of normal surfaces which intersect an
edge $[x_i, x_j]$ of $\Phi$ corresponds to those $w_i$ lying in the
geodesic $[t_i, t_j]$ in $\Xt$. The union of the three geodesics
$[t_i, t_j]$ forms a minimal spanning tree for the three cell
splitters in $\Xt$.  Let $t_0$ be the center of this tree, i.e. the
unique point that lies in all three geodesics. Note that the tree may
be degenerate, so $t_0$ may be equal to one of the other vertices.

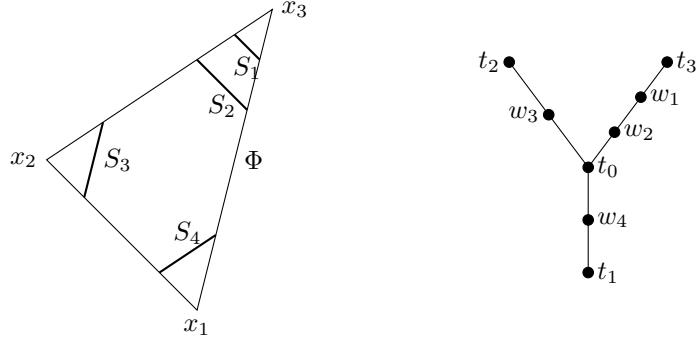
\begin{figure}[H]
\begin{center} 
\begin{tikzpicture}

\draw (0, 0) node [below] {$x_1$} -- 
      (1, 4) node [right] {$x_3$} node [midway, right] {$\Phi$}-- 
      (-2, 2) node [left] {$x_2$}-- cycle;

\draw [thick] ( $(0,0)!0.833!(1,4)$ ) -- node [midway, below] {$S_1$}
($(-2,2)!0.833!(1,4)$ );

\draw [thick] ( $(0,0)!0.666!(1,4)$ ) -- node [midway, below] {$S_2$}
($(-2,2)!0.666!(1,4)$ );

\draw [thick] ( $(-2,2)!0.25!(1,4)$ ) -- node
[midway, right] {$S_3$} ($(0,0)!0.75!(-2,2)$ );

\draw [thick] ( $(0,0)!0.25!(1,4)$ ) -- node [midway, above] {$S_4$}
($(0,0)!0.25!(-2,2)$ );

\begin{scope}[xshift=1cm, yshift=0.5cm, scale=1.4]
\filldraw[black] (3,0) circle (0.05cm) node [right]{$t_1$};
\filldraw[black] (2.25,2) circle (0.05cm) node [left]{$t_2$};
\filldraw[black] (3.75,2) circle (0.05cm) node [right]{$t_3$};
\filldraw[black] (3,1) circle (0.05cm) node [right]{$t_0$};
\filldraw[black] ( $(3.75,2)!0.333!(3, 1)$ ) circle (0.05cm) node [right] {$w_1$};
\filldraw[black] ( $(3.75,2)!0.666!(3, 1)$ ) circle (0.05cm) node [right] {$w_2$};
\filldraw[black] ( $(2.25,2)!0.5!(3, 1)$ ) circle (0.05cm) node [left] {$w_3$};
\filldraw[black] (3,0.5) circle (0.05cm) node [right]{$w_4$};

\draw (3,0) -- (3, 1) -- (3.75, 2);

\draw (2.25, 2) -- (3, 1); 

\end{scope}

\end{tikzpicture}%
\end{center} 
\caption{Example of normal surfaces intersecting a face of the dual
  triangulation.}
\label{pic:normal tree} 
\end{figure}

Normal arcs parallel to the edge $[x_2, x_3]$ correspond to surfaces
which hit both of the edges $[x_1, x_2]$ and $[x_1, x_3]$, so
correspond points $w_i$ which lie in both $[t_1, t_2]$ and $[t_1,
t_3]$, and similarly for the other two cases.  The intersection of
these two geodesics in $\Xt$ is $[t_1, t_0]$, and so the corresponding
surfaces appear in the same order on each of the edges in $\Phi$, and so
the arcs are disjoint. The same argument applies to each vertex of
$\Phi$.
\end{proof}

As the resulting surfaces in $\twid{M}$ are disjoint and equivariant,
they project down to disjoint surfaces in $M$.

We now show that the polyhedral surfaces, and their complements,
project down homeomorphically into $M$. As the level set surfaces lift
homeomorphically to $\twid{M}$, the area bound for the level sets of
$f$ is also an area bound for the level sets of $h$. Therefore, each
polyhedral surface contains a bounded number of faces.  The deck
transformation group of the universal cover of a graph is equal to the
fundamental group of the graph, which is a free group, so the orbit of
any face consists of infinitely many disjoint translates. If two lie
in the same connected component of a polyhedral surface, then that
path corresponds to a covering translation, which has infinite order,
so in fact the connected component contains infinitely many faces,
which contradicts the fact that there is a bound on the number of
faces in each component.

Each complementary region is compact, so the same argument applied to
the complementary regions shows that they are all mapped down
homeomorphically as well.

\subsection{Bounded handles}\label{section:bounded}

We now bound the number of handles in a complementary region of the
capped surfaces, which contains a single cell splitter.  The following
result will complete the proof of the left hand inequality in Theorem
\ref{theorem:graph}.

\begin{proposition}\label{bounded_handles} 
Let $(\mu, \e)$ be $MV$-constants, and let $M$ be a complete
hyperbolic $3$-manifold without cusps, with an $\e$-regular Voronoi
decomposition $\V$, and a generic Morse function $h \colon M \to Y$,
where $Y$ is a tree. Let $\{ u_i \}$ be a collection of points in $Y$,
which separate the cell splitters in $Y$, and let $\{ S^+_i \}$ be the
corresponding collection of capped surfaces.  If $P$ is a
complementary component of the capped surfaces in $M$, the region $P$
has at most three boundary components, $S^+_{i_1}, S^+_{i_2}$ and
$S^+_{i_3}$ say, where the final surface may be empty.  Then $P$ is
homeomorphic to a manifold with a handle decomposition containing at
most
\[ K \text{Gromov area} ( M ) \]
handles, where $K$ depends only on the $MV$-constants.
\end{proposition}

We start with the observation that attaching a compression body $P$ to
a $3$-manifold $Q$ by a subsurface $S$ of the upper boundary component
of $P$, requires a number of handles which is bounded in terms of the
Heegaard genus of $P$, and the number of boundary components of the
attaching surface.

\begin{proposition}\label{bounded_handles3}\cite{hoffoss-maher}*{Proposition 2.16}
Let $Q$ be a compact $3$-manifold with boundary, and let $R = Q \cup
P$, where $P$ is a compression body of genus $g$, attached to $Q$ by a
homeomorphism along a (possibly disconnected) subsurface $S$ contained
in the upper boundary component of $P$ of genus $g$. Then $R$ is
homeomorphic to a $3$-manifold obtained from $Q$ by the addition of at
most $(4g + 2 \norm{\partial S})$ $1$-and $2$-handles, where
$\norm{\partial S}$ is the number of boundary components of $S$.
\end{proposition}

\begin{proof}[Proof (of Proposition \ref{bounded_handles}).]
If $P$ has two boundary components, then the argument is exactly the
same as \cite{hoffoss-maher}*{Proposition 2.15}, so we now consider the
case in which $P$ has three boundary components, which, without loss
of generality we may relabel $S^+_{1}, S^+_{2}$ and $S^+_3$.  Let $t$
be the cell splitter corresponding to the region $P$, and let $V$ be
the corresponding Voronoi cell.  As $P$ has three boundary components,
$t$ must be a vertex of $Y$.

We first consider the case in which the Voronoi region $V$
corresponding to the cell splitter $t$ in $h(P)$ is disjoint from the
thin part $T_\mu$.  Consider the three polyhedral surfaces $S_1, S_2$
and $S_3$, corresponding to the three capped surfaces, and let
$\S = \cup S_i \cup V$ be the union of the polyhedral surfaces,
together with the Voronoi cell $V$.  By Proposition \ref{cor:bound},
there is a constant $K$, which only depends on the $MV$-constants,
such that the number of faces of $\S$ in the thick part is at most
$3 K_1 \text{Gromov area}(M)$, i.e.
\[ \| \S \cap X_\mu \| \le 3 K_1 \text{Gromov area}(M), \]
where $K_1$ is the constant from Proposition \ref{cor:bound}.  The
number of boundary components of each surface $S_i \cap X_\mu$ is also
bounded by Proposition \ref{cor:bound}, and by Proposition
\ref{prop:bound}, the Voronoi cell $V$ has a bounded number $J$ of
vertices, edges and faces, where $J$ depends only on the
$MV$-constants.  In particular, there is a constant $A$, depending
only on the $MV$-constants, such that $P \cap X_\mu$ has a handle
structure with at most $A ( \text{Gromov area}(M) )$ handles.

To bound the number of handles contained in $P$, we observe that $P$
is a regular neighbourhood of the $3$-complex obtained from capping
off the boundary components of $\S \cap X_\mu$, using the parts of the
capped surfaces in the thin part, i.e. the union of the components of
$S^+_i \cap T_\mu$ over all three capped surfaces.  Each component of
$S^+_i \cap T_\mu$ has genus at most one, and the number of boundary
components of $\S \cap X_\mu$ is bounded linearly in terms of
$\text{Gromov area}(M)$, therefore, there is a constant $B$, depending
only on the $MV$-constants, such that the number of handles in $P$ is
at most $B ( \text{Gromov area}(M) )$, as required.  

We now consider the case in which the region $P$ has image $h(P)$ in
$Y$ which contains the cell splitter $t$, and the corresponding
Voronoi cell $V$ intersects $T_\mu$.  In this case, the connected
components of $V \cap X_\mu$ need not be topological balls, and there
may be connected components of $P \cap T_\mu$ whose boundary components
are not parallel.

The connected components of $V \cap X_\mu$ are handlebodies of bounded
genus, as show in the following result of Kobayashi and Rieck
\cite{kobayashi-rieck}.  We state a simplified version of their result
which suffices for our purposes, see \cite{hoffoss-maher} for further
details.

\begin{proposition}\cite{kobayashi-rieck}
Let $\mu$ be a Margulis constant for $\mathbb{H}^3$, $M$ be a finite
volume hyperbolic $3$-manifold, let $0 < \e < \mu$, and let $\V$ be a
regular Voronoi decomposition of $M$ arising from a maximal collection
of $\e$-separated points. Then there is a number $G$, depending only
on $\mu$ and $\e$, such that for any Voronoi cell $V_i$, there are at
most $G$ connected components of $V_i \cap X_\mu$, each of which is a
handlebody of genus at most $G$, attached to $T_\mu$ by a surface with
at most $G$ boundary components.
\end{proposition}

Recall that attaching a handlebody of genus $G$ to a $3$-manifold
along a subsurface of the boundary with at most $G$ boundary
components requires at most $6G$ handles:

\begin{proposition}\cite{hoffoss-maher}*{Proposition 2.16}
Let $Q$ be a compact $3$-manifold with boundary and let $R = Q \cup P$
, where $P$ is a compression body of genus $g$, attached to $Q$ by a
homeomorphism along a (possibly disconnected) subsurface $S$ contained
in the upper boundary component of $P$ of genus $g$. Then $R$ is
homeomorphic to a $3$-manifold obtained from $Q$ by the addition of at
most $(4 \text{genus} + 2 \norm{ \partial S})$ $1$- and $2$-handles,
where $\norm{\partial S}$ is the number of boundary components of $S$.
\end{proposition}

Therefore, adding a Voronoi cell which intersects $\partial T_\mu$ may
be realized by at most $6 G^2$ handles.

If the Voronoi cell intersects $T_\mu$, then there may be components
of $P \cap T_\mu$ whose boundary surfaces are not parallel.  This case
is considered in the proof of \cite{hoffoss-maher}*{Proposition 2.15},
when the manifold has no infinite Margulis tubes, so it suffices to
consider the case of a component of $P$ contained in an infinite
Margulis tube.  However, the case of an infinite Margulis tube in
which neither surface is an essential disc is the same as the ordinary
Margulis tube case, and if both surfaces essential discs then they are
parallel.  Finally, if exactly one surface is an essential disc, then
the other surface lies in the same homology class, via the component
of $P$ in the infinite Margulis tube, and so, after surgering
inessential boundary components, is also an essential disc.  However,
the number of boundary components is at most $K_1 \text{Gromov
  area}(M)$, and so the total number of extra handles over all
components of $P$ in the infinite Margulis tubes is also bounded by
$K_1 \text{Gromov area}(M)$.

We may choose the constant $K$ to be the maximum of the constants
arising from the two cases considered above, thus completing the proof
of Proposition \ref{bounded_handles3}.
\end{proof}

\section{Topological complexity bounds metric complexity}\label{section:sss}

In this section we will show bounds for metric complexity in terms of
topological complexity, i.e. the right hand inequality in Theorem
\ref{theorem:graph}, assuming the Pitts and Rubinstein result, Theorem
\ref{conjecture:pr}.

We start by reminding the reader of the topological properties of thin
position for generalized Heegaard splittings, as shown by Scharlemann
and Thompson \cite{st} for the linear case and Saito, Scharlemann and
Schultens \cite{sss} for the graph case.

\begin{theorem} \cites{st, sss} \label{theorem:sss} %
Let $H$ be a graph splitting that is in thin position. Then every even
surface is incompressible in $M$ and the odd surfaces form strongly
irreducible Heegaard surfaces for the components of $M$ cut along the
even surfaces.
\end{theorem}

We will use the following result due to Gabai and Colding
\cite{colding-gabai}*{Appendix A}, building on recent work of Colding
and Minicozzi \cite{colding-minicozzi}.  It is not stated explicitly
in their paper, but see \cite{hoffoss-maher}*{Theorem 3.2} for further
details.

\begin{theorem} \cite{colding-gabai} \label{theorem:minimal}
Let $M$ be a hyperbolic manifold, with (possibly empty) least area
boundary, with a minimal Heegaard splitting $H$ of genus $g$. Then,
assuming Theorem \ref{conjecture:pr}, the manifold $M$ has a
(possibly singular) foliation by compact leaves, containing the
boundary surfaces as leaves, such that each leaf has area at most $4
\pi g$.
\end{theorem}

By Theorem \ref{theorem:sss}, we may consider the compression bodies in the
graph splitting in pairs, glued along strongly irreducible Heegaard
splittings, and then Theorem \ref{theorem:minimal} guarantees that
each pair has a foliation with each leaf having area at most $4 \pi
g$.  These foliations contain the boundary surfaces as leaves, and so
the foliations on each pair extend to foliations of the entire
manifold, as required.



\bigskip

\noindent Diane Hoffoss \\
University of San Diego \\
\url{dhoffoss@sandiego.edu} \\

\noindent Joseph Maher \\
CUNY College of Staten Island and CUNY Graduate Center \\
\url{joseph.maher@csi.cuny.edu} \\


\end{document}